\documentclass[reqno]{amsart}

\usepackage[margin=1in]{geometry}

\usepackage{mathtools}
\usepackage{amssymb, amsfonts}
\usepackage{mathrsfs}
\usepackage{microtype}
\usepackage{enumitem}

\numberwithin{equation}{section}

\theoremstyle{plain}
\newtheorem{theorem}{Theorem}[section]
\newtheorem{lemma}[theorem]{Lemma}

\theoremstyle{definition}

\theoremstyle{remark}

\usepackage[hidelinks]{hyperref}

\title[A Sharper Explicit Result for the Sum of Two Almost Primes]
{A Sharper Explicit Result for the Sum of Two Almost Primes}

\author{Peter J. Campbell}

\address{School of Mathematics and Physics, The University of Queensland, St Lucia, Brisbane, QLD 4072, Australia}
\email{p.campbell1@uq.edu.au}



\begin{document}

\begin{abstract}
    We prove that for every integer \(N\geq2\), there exist positive integers \(a\) and \(b\) such that \(N=a+b\) and \(\Omega(ab)\leq33\), where \(\Omega(n)\) denotes the number of prime factors of \(n\), counted with multiplicity. This improves the previous bound of \(40\) obtained by Dudek and Dunn. The proof applies the explicit Friedlander--Iwaniec \(\Lambda^-\Lambda^2\) lower-bound sieve to a sequence derived from the products \(n(N-n)\). The main new ingredient is pre-sieving at the prime \(3\), which eliminates the extremal small-prime case in the dimension condition while keeping the resulting remainder terms under explicit control. We complete the proof using analytic estimates for large \(N\), finite verification over an intermediate range, and explicit prime-gap data for small \(N\).
\end{abstract}

\maketitle

\section{Introduction}

Goldbach's conjecture is one of the oldest and most famous open problems in additive number theory. In its strong form, it asserts that every even integer greater than two can be written as the sum of two primes. Although the conjecture remains unresolved, it has motivated extensive work on representing integers as sums of numbers with few prime factors. Such numbers are commonly
called \textit{almost primes}. Throughout this paper, \(\Omega(n)\) denotes the
number of prime factors of \(n\), counted with multiplicity.

A landmark result in this direction is Chen's theorem, which states that every sufficiently large even integer can be written as the sum of a prime and an integer with at most two prime factors~\cite{chen1973large_even_integer}. Whereas Chen's theorem applies only to sufficiently large even integers, a classical theorem of R\'enyi~\cite{renyi1948_even_number} is uniform in \(N\): there is a fixed natural number \(C\) such that every even integer \(N\geq4\) is the sum of a prime and an integer with at most \(C\) prime factors. More recently, Johnston and Starichkova~\cite{johnston_starichkova2025_sum_prime_almost_prime} proved the following explicit form of Rényi's theorem.

\begin{theorem}\label{thm:JS}
    Every even integer \(N\geq4\) can be written as
    \[
    N=p+r,
    \]
    where \(p\) is prime and \(r\) is a positive integer satisfying
    \[
    \Omega(r)\leq395.
    \]
\end{theorem}

Their proof uses the explicit linear sieve of Bordignon, Johnston, and Starichkova~\cite{bordignon_johnston_starichkova2025}. They also prove, assuming the generalised Riemann hypothesis, the sharper bound \(\Omega(r)\leq31\).

In this paper, we relax the requirement that one of the two summands be prime. Rather than seeking a representation
\[
N=p+r
\]
with \(p\) prime and \(r\) almost prime, we seek a representation
\[
N=a+b
\]
for which the product \(ab\) has few prime factors. It is natural to ask how much this weaker requirement reduces the number of prime factors needed.

The first explicit result for this problem was obtained by Dudek and Dunn~\cite{dudek_dunn2026_sum_two_almost_primes}, who proved that every integer \(N\geq2\) can be written as \(N=a+b\), where \(a\) and \(b\) are positive integers satisfying
\[
\Omega(ab)\leq40.
\]
Their argument combines the explicit Friedlander--Iwaniec \(\Lambda^-\Lambda^2\) lower-bound sieve~\cite{friedlander_iwaniec2010opera} with estimates tailored to the sequence under consideration, a finite computation, and an explicit treatment of small values of \(N\).

Our main result is the following improvement of the Dudek--Dunn bound.

\begin{theorem}\label{thm:main}
Every integer \(N\geq2\) can be written as
\[
N=a+b,
\]
where \(a\) and \(b\) are positive integers satisfying
\[
\Omega(ab)\leq33.
\]
\end{theorem}

The main new ingredient is pre-sieving at the prime \(3\). In the Dudek--Dunn application of the Friedlander--Iwaniec lower-bound sieve, the verification of the dimension condition is constrained by an extremal small-prime case arising from the contribution at \(p=3\). By retaining only those indices \(n\) for which
\[
3\nmid n(N-n),
\]
and correspondingly omitting the prime \(3\) from the sifting product, we eliminate this extremal case. This permits a substantially smaller constant in the dimension condition, while the corresponding change in the remainder terms is controlled through an inclusion--exclusion identity for the pre-sieved sequence. We also track several numerical estimates more carefully than is necessary for the Dudek--Dunn bound of \(40\), allowing the final bound to be reduced to \(33\).

The remainder of the paper is organised as follows. In Section~\ref{sec:sieve-setup}, we introduce the required sieve notation and recall the explicit Friedlander--Iwaniec lower-bound sieve in the form used here. In Section~\ref{sec:presieve}, we construct the pre-sieved sequence, identify the main term in the associated counting formula, and derive the corresponding remainder identity. In Section~\ref{sec:estimates}, we establish the required dimension conditions and bounds for the main-term factor and the accumulated error term. Finally, in Section~\ref{sec:proof-main}, we combine these estimates to prove Theorem~\ref{thm:main}: explicit prime-gap data handle the small range, finite verification handles the intermediate range, and analytic estimates handle the large range.

\section{Sieve setup}\label{sec:sieve-setup}

In this section, we record the general form of the Friedlander--Iwaniec lower-bound sieve used in the proof. We formulate the result for a finite indexed sequence and an arbitrary set of sifting primes, so that it may later be applied directly to the sequence obtained by pre-sieving at the prime \(3\).

Let
\[
\mathcal A=(a_n)_{n\in I}
\]
be a finite indexed sequence of positive integers, and let \(\mathcal P\) be a set of sifting primes. For square-free \(d\) composed of primes in \(\mathcal P\), let
\[
\mathcal A_d:=(a_n)_{\substack{n\in I\\ d\mid a_n}}.
\]
Suppose that, for every such \(d\),
\[
|\mathcal A_d|=Xg(d)+r_d,
\]
where \(X>0\), \(g\) is multiplicative on the square-free integers composed of primes in \(\mathcal P\), and
\[
0<g(p)<1
\qquad (p\in\mathcal P).
\]

For \(z\geq2\), define the sifting product
\[
P(z):=\prod_{\substack{p<z\\ p\in\mathcal P}}p
\]
and the corresponding sifted sum
\[
S(\mathcal A,z)
:=
\#\{n\in I:(a_n,P(z))=1\}.
\]
The associated main-term factor is
\[
V(z):=\prod_{p\mid P(z)}(1-g(p)).
\]

Let \(\kappa\geq1\) and \(K\geq1\). We say that \(g\) satisfies the dimension condition with parameters \(\kappa\) and \(K\) if, for each value of \(z\) under consideration and every \(2\leq w<z\),
\begin{equation}\label{sieveDimension}
    \prod_{\substack{w\leq p<z\\ p\in\mathcal P}}
    (1-g(p))^{-1}
    \leq
    K\left(\frac{\log z}{\log w}\right)^\kappa.
\end{equation}
The parameter \(\kappa\) is called the dimension of the sieve.

We shall use the following form of the Friedlander--Iwaniec lower-bound sieve, obtained from Theorem~7.7 of~\cite{friedlander_iwaniec2010opera} within the general framework of Theorem~7.1 of the same book.

\begin{theorem}[Friedlander--Iwaniec lower-bound sieve]\label{thm:FI}
    With the notation and assumptions above, let \(D\geq z\), and put
    \[
    s:=\frac{\log D}{\log z},
    \qquad
    k:=\kappa+\log K.
    \]
    If \(s \geq 2k+3\), then
    \[
    S(\mathcal A,z)
    \geq
    XV(z)F(s,k)-2R_4(\mathcal A,D),
    \]
    where the sieve factor \(F(s,k)\) is defined by
    \begin{equation}\label{eq:sieve-factor}
    F(s,k)
    :=
    1-\frac{s+3}{2e^k}
    \left(\frac{2ek}{s-3}\right)^{(s-3)/2},
    \end{equation}
    and the accumulated error term is
    \[
    R_4(\mathcal A,D)
    :=
    \sum_{\substack{d\mid P(z)\\ d<D}}
    \tau_4(d)|r_d|,
    \]
    where
    \[
    \tau_4(n)
    :=
    \#\bigl\{(d_1,d_2,d_3,d_4)\in\mathbb Z_{>0}^4:
    d_1d_2d_3d_4=n\bigr\}.
    \]
\end{theorem}

\section{The pre-sieved sequence}\label{sec:presieve}

Throughout Sections~\ref{sec:presieve}--\ref{sec:estimates}, let \(N\geq2\) be fixed. We first construct an auxiliary indexed sequence, from which the pre-sieved sequence and the sifting data used in Theorem~\ref{thm:FI} will be obtained.

For \(1\leq n\leq N-1\), define
\[
a_n =
\begin{cases}
    n(N-n), & \text{if }N\text{ is even},\\[4pt]
    \dfrac{n(N-n)}{2}, & \text{if }N\text{ is odd}.
\end{cases}
\]
Let the auxiliary sequence be
\[
\mathcal B=\mathcal B(N):=(a_n)_{1\leq n\leq N-1}
\]
and put
\[
Y:=N-1.
\]
We regard \(\mathcal B\) as an indexed sequence, so that the repeated values arising from the symmetry \(a_n=a_{N-n}\) are retained with their natural multiplicities.

We now define a multiplicative local density function \(g\) on the square-free integers. If \(N\) is even, then \(2\mid a_n\) precisely when \(n\) is even. If \(N\) is odd, then exactly one of \(n\) and \(N-n\) is even, and \(2\mid a_n\) precisely when this even factor is divisible by \(4\). Equivalently,
\[
n\equiv0\pmod4
\qquad\text{or}\qquad
n\equiv N\pmod4.
\]
In either case, the corresponding local density is
\[
g(2)=\frac12.
\]

For an odd prime \(p\), division by two is invertible modulo \(p\), so the condition \(p\mid a_n\) is equivalent to
\[
n\equiv0\pmod p
\qquad\text{or}\qquad
n\equiv N\pmod p.
\]
These residue classes are distinct when \(p\nmid N\) and coincide when \(p\mid N\). We therefore define
\[
g(p)=
\begin{cases}
\dfrac12, & p=2,\\[6pt]
\dfrac2p, & p\neq2\text{ and }p\nmid N,\\[6pt]
\dfrac1p, & p\neq2\text{ and }p\mid N,
\end{cases}
\]
and extend \(g\) multiplicatively to the square-free integers.

For square-free \(d\), let
\[
\mathcal B_d
:=
(a_n)_{\substack{1\leq n\leq N-1\\ d\mid a_n}}
\]
and define \(\rho_d\) by
\begin{equation}\label{eq:ambient-count}
    |\mathcal B_d|
    =
    Yg(d)+\rho_d.
\end{equation}

We now pre-sieve at the prime \(3\). Let
\[
I_3:=\{1\leq n\leq N-1:3\nmid a_n\}
\]
and define
\[
\mathcal A=\mathcal A(N):=(a_n)_{n\in I_3}.
\]
Put
\[
X:=|I_3|=|\mathcal A|.
\]
Since \(|\mathcal B|=Y\), applying \eqref{eq:ambient-count} with \(d=3\) gives
\begin{equation}\label{eq:presieved-size}
    X = Y(1-g(3))-\rho_3.
\end{equation}

A direct count modulo \(3\) also gives
\[
X =
\begin{cases}
    \dfrac{2N}{3}, & N\equiv0\pmod3,\\[6pt]
    \dfrac{N-1}{3}, & N\equiv1\pmod3,\\[6pt]
    \dfrac{N+1}{3}, & N\equiv2\pmod3.
\end{cases}
\]

We take
\[
\mathcal P:=\{p:p\text{ is prime and }p\neq3\}
\]
as the set of sifting primes. For square-free \(d\) composed of primes in \(\mathcal P\), define
\[
\mathcal A_d
:=
(a_n)_{\substack{n\in I_3\\ d\mid a_n}}.
\]
Since \(3\nmid d\), inclusion--exclusion over the pre-sieved prime \(3\) gives
\[
|\mathcal A_d|
=
|\mathcal B_d|-|\mathcal B_{3d}|.
\]

Using \eqref{eq:ambient-count} and the multiplicativity of \(g\), we obtain
\begin{align*}
|\mathcal A_d|
&=
|\mathcal B_d|-|\mathcal B_{3d}|\\
&=
Yg(d)+\rho_d-Yg(3d)-\rho_{3d}\\
&=
Y(1-g(3))g(d)+\rho_d-\rho_{3d}.
\end{align*}
On the other hand, \eqref{eq:presieved-size} gives
\[
Y(1-g(3))=X+\rho_3.
\]
Consequently,
\[
|\mathcal A_d|
=
Xg(d)+r_d,
\]
where
\begin{equation}\label{eq:presieved-remainder}
r_d
:=
\rho_d-\rho_{3d}+g(d)\rho_3.
\end{equation}
Thus the pre-sieved sequence has the same local density \(g(d)\) at square-free moduli composed of primes in \(\mathcal P\), but with the modified remainder terms \(r_d\). Since every term of \(\mathcal A\) is already coprime to \(3\), for
\(z>3\) the condition
\[
(a_n,P(z))=1
\]
is equivalent to \(a_n\) having no prime divisor less than \(z\).

\section{Explicit estimates for the pre-sieved sequence}\label{sec:estimates}
We now record the explicit estimates needed in the proof of Theorem~\ref{thm:main}. That proof will be divided into three ranges. Define
\[
N_0:=117465180365547648498934439.
\]
For \(N>N_0\), the proof will use the sifting parameter
\[
z=N^{1/16.5}.
\]
We also define
\[
z_0:=N_0^{1/16.5}=38.0184\dots,
\qquad
z_1:=10^8,
\qquad
N_1:=z_1^{16.5}=10^{132}.
\]
The three ranges are as follows:
\begin{enumerate}
    \item \(2\leq N\leq N_0\). This initial finite range will be handled directly by Lemma~\ref{lem:smallN}.
    \item \(N_0<N<N_1\), equivalently \(z_0<z<z_1\). In this range, positivity of the pre-sieved sum will be established by finite verification.
    \item \(N\geq N_1\), equivalently \(z\geq z_1\). In this range, positivity will follow from explicit analytic estimates.
\end{enumerate}

The estimates in this section are organised according to this division. We first obtain bounds valid for all \(z\geq z_0\), which are used in the finite verification range, and then sharper estimates valid for \(z\geq z_1\), which are used in the analytic range.

The estimates below differ from those of Dudek and Dunn in three main respects. In their setting, the extremal case \(w=3\) in the dimension condition requires the choice \(K=3\). Pre-sieving removes the prime \(3\) and hence eliminates this case, allowing a substantially smaller value of \(K\) in the present argument. We also use the restrictions \(z\geq z_0\) and, in the analytic range, \(z\geq z_1\), rather than estimates uniform over a larger range. Finally, several tail bounds and worst-case substitutions are replaced by sharper analytic estimates or finite verification.

We first require explicit upper bounds for products of the form
\[
\prod_{\substack{w\leq p<z\\ p\in\mathcal P}}
(1-g(p))^{-1}.
\]
The following lemma gives both the estimate used in the finite verification range and the sharper estimate used in the analytic range. Although the eventual application takes \(z=N^{1/16.5}\), in the following lemma \(z\) is regarded as an independent real variable.

\begin{lemma}\label{lem:K-presieved}
    The following estimates hold.
    \begin{enumerate}
        \item For every \(z\geq z_0\) and every \(2\leq w<z\),
        \[
        \prod_{\substack{w\leq p<z\\ p\in\mathcal P}}
        (1-g(p))^{-1}
        \leq
        1.146
        \left(\frac{\log z}{\log w}\right)^2.
        \]
        
        \item For every \(z\geq z_1\) and every \(2\leq w<z\),
        \[
        \prod_{\substack{w\leq p<z\\ p\in\mathcal P}}
        (1-g(p))^{-1}
        \leq
        1.097
        \left(\frac{\log z}{\log w}\right)^2.
        \]
    \end{enumerate}
\end{lemma}

\begin{proof}
    The estimates are uniform in \(N\), since the individual factors admit bounds independent of \(N\). Indeed,
    \[
    (1-g(2))^{-1}=2,
    \]
    while, for odd primes \(p\neq3\),
    \[
    (1-g(p))^{-1}
    \leq
    \left(1-\frac{2}{p}\right)^{-1}.
    \]
    The numerical computations below can be reproduced using the verification repository \cite{VerificationRepo}. The relevant script is \path{verify_K_presieved.py}.
    
    The threshold \(286\) in the case division below is the validity threshold for the upper bound in Theorem~5 of Rosser and Schoenfeld~\cite{rosser_schoenfeld1962} used to estimate the prime harmonic sum.

    \medskip
    
    \noindent
    We first prove assertion~\textup{(1)}.
    
    \medskip
    
    \noindent
    \textbf{Case 1: \(w\geq286\).}
    If \(z\leq293\), then the product is empty, since there are no primes in the interval \((283,293)\). Hence the desired estimate is immediate. We may therefore suppose that \(z>293\). Put
    \[
    y:=\max\{w,293\}.
    \]
    Using the Taylor expansion of \(-\log(1-u)\), we obtain
    \begin{align}\label{eq:log-product-bound}
        \log\left(
        \prod_{\substack{w\leq p<z\\ p\in\mathcal P}}
        (1-g(p))^{-1}
        \right)
        &\leq
        \sum_{y\leq p<z}
        -\log\left(1-\frac{2}{p}\right) \notag\\
        &=
        2\sum_{y\leq p<z}\frac{1}{p} + \sum_{y\leq p<z} \sum_{j=2}^{\infty}\frac1j\left(\frac2p\right)^j \notag\\
        &\leq
        2\sum_{y\leq p<z}\frac{1}{p} + 2\sum_{y\leq p<z}\frac{1}{p(p-2)}.
    \end{align}
    By Theorem~5 of Rosser and Schoenfeld,
    \begin{align*}
        2\sum_{y\leq p<z}\frac{1}{p}
        &<
        2\log\left(\frac{\log z}{\log y}\right) + \frac{1}{\log^2 z} + \frac{1}{\log^2 y} \\
        &\leq
        2\log\left(\frac{\log z}{\log w}\right) + \frac{2}{\log^2 293}.
    \end{align*}
    For the remaining sum, we use \(\zeta(2)=\pi^2/6\) to obtain
    \begin{align}\label{eq:zeta-2}
        \sum_{y\leq p<z}\frac{1}{p(p-2)}
        &\leq
        \sum_{293\leq p<100000}\frac{1}{p(p-2)} + \sum_{n\geq99998}\frac{1}{n^2} \notag\\
        &=
        \sum_{293\leq p<100000}\frac{1}{p(p-2)} + \frac{\pi^2}{6} - \sum_{n=1}^{99997}\frac{1}{n^2}.
    \end{align}
    A direct computation gives
    \[
    \frac{2}{\log^2 293} + 2\left(\sum_{293\leq p<100000}\frac{1}{p(p-2)} + \frac{\pi^2}{6} - \sum_{n=1}^{99997}\frac{1}{n^2}\right)
    <
    \log(1.0651).
    \]
    Consequently,
    \[
    \prod_{\substack{w\leq p<z\\ p\in\mathcal P}} (1-g(p))^{-1}
    \leq
    1.0651 \left(\frac{\log z}{\log w}\right)^2.
    \]
    
    \medskip
    
    \noindent
    \textbf{Case 2: \(2\leq w<286\) and \(z\geq286\).}
    Since there are no primes in the interval \((283,293)\), we split the product as
    \[
    \prod_{\substack{w\leq p<z\\ p\in\mathcal P}} (1-g(p))^{-1}
    \leq
    \left(\prod_{\substack{w\leq p\leq283\\ p\in\mathcal P}} (1-g(p))^{-1}\right) \left(\prod_{293\leq p<z} \left(1-\frac{2}{p}\right)^{-1}\right),
    \]
    where the second product is interpreted as the empty product when
    \(286\leq z\leq293\).
    
    The estimate
    \[
    \prod_{293\leq p<z}\left(1-\frac2p\right)^{-1}
    \leq
    1.0651\left(\frac{\log z}{\log293}\right)^2
    \]
    holds throughout \(z\geq286\): for \(286\leq z\leq293\), the product
    is empty and the right-hand side exceeds \(1\), while for \(z>293\)
    it follows from Case~1 with lower endpoint \(293\). Therefore
    \begin{align*}
        \prod_{\substack{w\leq p<z\\ p\in\mathcal P}} (1-g(p))^{-1}
        &\leq
        1.0651 \left(\prod_{\substack{w\leq p\leq283\\ p\in\mathcal P}} (1-g(p))^{-1}\right) \left(\frac{\log z}{\log293}\right)^2 \\
        &=
        1.0651 \left(\frac{\log w}{\log293}\right)^2 \left(\prod_{\substack{w\leq p\leq283\\ p\in\mathcal P}} (1-g(p))^{-1}\right) \left(\frac{\log z}{\log w}\right)^2.
    \end{align*}
    A finite computation, using the upper bounds for the factors stated at the beginning of the proof, gives
    \[
    \sup_{2\leq w<286} \left\{1.0651 \left(\frac{\log w}{\log293}\right)^2 \prod_{\substack{w\leq p\leq283\\ p\in\mathcal P}} (1-g(p))^{-1} \right\}
    \leq
    1.1282.
    \]
    It is enough to check \(w=2\), the prime endpoints \(5\leq w\leq283\), and the limiting endpoint \(w\to286^{-}\). Indeed, the finite product is constant between consecutive primes, while \((\log w)^2\) is increasing, and the only jumps in the product occur when \(w\) passes a prime.
    
    Thus
    \[
    \prod_{\substack{w\leq p<z\\ p\in\mathcal P}} (1-g(p))^{-1}
    \leq
    1.1282 \left(\frac{\log z}{\log w}\right)^2.
    \]
    
    \medskip
    
    \noindent
    \textbf{Case 3: \(z_0\leq z<286\).}
    For a fixed set of included primes, the quantity
    \[
    \prod_{\substack{w\leq p<z\\ p\in\mathcal P}} (1-g(p))^{-1} \left(\frac{\log w}{\log z}\right)^2
    \]
    is decreasing as a function of \(z\). Thus, on each interval between consecutive primes, its largest value is approached at the left endpoint. The only points at which the set of included primes can change are the primes themselves. Hence the supremum over \(z_0\leq z<286\) is obtained among \(z=z_0\) and the right-hand limiting values \(z\to q^{+}\), where \(q\) runs over the primes satisfying \(z_0\leq q<286\).
    
    For each such value of \(z\), it is enough to check \(w=2\) and the prime endpoints \(5\leq w<z\), since the finite product is constant between consecutive primes as a function of \(w\), whereas \((\log w)^2\) is increasing.
    
    A finite computation, again using the upper bounds for the factors stated at the beginning of the proof, gives
    \[
    \sup_{\substack{z_0\leq z<286\\2\leq w<z}} \left\{\prod_{\substack{w\leq p<z\\ p\in\mathcal P}} (1-g(p))^{-1} \left(\frac{\log w}{\log z}\right)^2 \right\}
    \leq
    1.1458.
    \]
    The largest value in this finite check occurs in the right-hand limit \(z\to43^{+}\), with \(w=5\). Since \(1.1458<1.146\), Cases~1–3 prove assertion~\textup{(1)}.
    
    \medskip
    
    \noindent
    We now prove assertion~\textup{(2)}. Since \(z\geq z_1>293\), only the analogues of Cases~1 and~2 are required.
    \medskip
    
    \noindent
    \textbf{Case 4: \(w\geq286\) and \(z \geq z_1\).}
    We begin from inequality~\eqref{eq:log-product-bound} from Case~1. Again employing Theorem~5 of Rosser and Schoenfeld~\cite{rosser_schoenfeld1962}, we have
    \begin{align*}
        2\sum_{y\leq p<z}\frac{1}{p}
        &<
        2\log\left(\frac{\log z}{\log y}\right) + \frac{1}{\log^2 z} + \frac{1}{\log^2 y} \\
        &\leq
        2\log\left(\frac{\log z}{\log w}\right) + \frac{1}{\log^2 10^8} + \frac{1}{\log^2 293}.
    \end{align*}
    Combining this estimate with \eqref{eq:zeta-2}, a direct computation gives
    \[
    \frac{1}{\log^2 10^8} + \frac{1}{\log^2 293} + 2\left(\sum_{293\leq p<100000}\frac{1}{p(p-2)} + \frac{\pi^2}{6} - \sum_{n=1}^{99997}\frac{1}{n^2}\right)
    <
    \log(1.0356).
    \]
    Consequently,
    \[
    \prod_{\substack{w\leq p<z\\ p\in\mathcal P}} (1-g(p))^{-1}
    \leq
    1.0356 \left(\frac{\log z}{\log w}\right)^2.
    \]

    \medskip
    
    \noindent
    \textbf{Case 5: \(2\leq w<286\) and \(z\geq z_1\).}
    Again, using the fact that there are no primes in the interval \((283,293)\), we split the product as
    \[
    \prod_{\substack{w\leq p<z\\ p\in\mathcal P}} (1-g(p))^{-1}
    \leq
    \left(\prod_{\substack{w\leq p\leq283\\ p\in\mathcal P}} (1-g(p))^{-1}\right) \left(\prod_{293\leq p<z} \left(1-\frac{2}{p}\right)^{-1}\right).
    \]
    
    The estimate from Case~4 gives
    \[
    \prod_{293\leq p<z} \left(1-\frac{2}{p}\right)^{-1}
    \leq
    1.0356 \left(\frac{\log z}{\log293}\right)^2.
    \]
    Therefore
    \begin{align*}
        \prod_{\substack{w\leq p<z\\ p\in\mathcal P}} (1-g(p))^{-1}
        &\leq
        1.0356 \left(\prod_{\substack{w\leq p\leq283\\ p\in\mathcal P}} (1-g(p))^{-1}\right) \left(\frac{\log z}{\log293}\right)^2 \\
        &=
        1.0356 \left(\frac{\log w}{\log293}\right)^2 \left(\prod_{\substack{w\leq p\leq283\\ p\in\mathcal P}} (1-g(p))^{-1}\right) \left(\frac{\log z}{\log w}\right)^2.
    \end{align*}
    A finite computation, using the same finite check used in Case~2, gives
    \[
    \sup_{2\leq w<286} \left\{1.0356 \left(\frac{\log w}{\log293}\right)^2 \prod_{\substack{w\leq p\leq283\\ p\in\mathcal P}} (1-g(p))^{-1} \right\}
    \leq
    1.0969.
    \]

    The largest value occurs at \(w=5\). Consequently,
    \[
    \prod_{\substack{w\leq p<z \\ p\in\mathcal P}} (1-g(p))^{-1}
    \leq
    1.0969 \left(\frac{\log z}{\log w}\right)^2.
    \]
    
    From Cases~4 and 5, assertion~\textup{(2)} holds.

    This completes the proof.
\end{proof}

We next require a lower bound for the main-term factor \(V(z)\). Since the prime \(3\) has been omitted from the sifting set, this factor is larger than the analogous factor in the setting without pre-sieving. We also need the estimate to hold from the cutoff \(z_1=10^8\), rather than the cutoff \(10^{10}\) used in the corresponding estimate of Dudek and Dunn. The following lemma supplies the required lower bound.

\begin{lemma}\label{lem:V-presieved}
    For \(z\geq z_1\), we have
    \[
    V(z)\geq \frac{1.241}{\log^2 z}.
    \]
\end{lemma}

\begin{proof}
    Since \(g(2)=1/2\) and \(g(p)\leq2/p\) for every odd prime \(p\), we have
    \[
    V(z)
    =
    \prod_{\substack{p<z\\ p\in\mathcal P}}(1-g(p))
    \geq
    \frac{1}{2} \prod_{5\leq p<z} \left(1-\frac{2}{p}\right)
    = 
    \frac{1}{2} \prod_{5\leq p < 10^8} \left(1-\frac{2}{p}\right) \prod_{10^8\leq p<z} \left(1-\frac{2}{p}\right).
    \]
    
    Using the Taylor expansion of \(-\log(1-u)\), we have
    \[
    -\log\left(1-\frac{2}{p}\right)
    =
    \frac{2}{p} + \sum_{j=2}^{\infty}\frac1j\left(\frac2p\right)^j
    \leq
    \frac{2}{p} + \frac{2}{p(p-2)}.
    \]
    Summing over \(10^8\leq p<z\), we obtain
    \begin{equation}\label{eq:logprod}
    -\log \prod_{10^8\leq p<z}\left(1-\frac{2}{p}\right)
    \leq
    2\sum_{10^8\leq p<z}\frac{1}{p} + \sum_{10^8\leq p<z}\frac{2}{p(p-2)}.
    \end{equation}
    For the first sum, Theorem~5 of Rosser and Schoenfeld~\cite{rosser_schoenfeld1962} gives
    \[
    \sum_{10^8\leq p<z}\frac{1}{p}
    <
    \log \log z - \log \log 10^8 + \frac{1}{2 \log^2 z} + \frac{1}{2 \log^2 10^8}.
    \]
    Since \(\log^2 z\geq\log^2 10^8\), it follows that
    \begin{align*}
    \sum_{10^8\leq p<z}\frac1p
    &<
    \log\log z
    -
    \left(
    \log\log 10^8-\frac{1}{\log^2 10^8}
    \right)\\
    &=
    \log\log z-2.9105269\dots.
    \end{align*}
    For the second sum in \eqref{eq:logprod}, we have the bound
    \[
    \sum_{10^8\leq p<z}\frac{2}{p(p-2)}
    \leq
    \sum_{n\geq 10^8}\frac{2}{n(n-2)}
    =
    \sum_{n\geq 10^8}\left(\frac{1}{n-2}-\frac{1}{n}\right)
    =
    \frac{1}{10^8-2}+\frac{1}{10^8-1}
    <
    10^{-7}.
    \]

    Therefore,
    \[
    -\log\left(\prod_{10^8 \leq p<z}\left(1-\frac{2}{p}\right)\right)
    <
    2\log\log z - 5.82105 + 10^{-7}
    <
    2\log\log z - 5.821.
    \]
    Equivalently,
    \[
    \prod_{10^8 \leq p<z}\left(1-\frac{2}{p}\right)
    >
    \frac{e^{5.821}}{\log^2 z}.
    \]
    
    A finite computation, reproduced by the script \path{verify_V_presieved_lower_bound.py} in the verification repository~\cite{VerificationRepo}, gives
    \[
    \frac{1}{2} \prod_{5\leq p < 10^8} \left(1-\frac{2}{p}\right)
    =
    0.00367979\ldots
    >
    0.0036797.
    \]
    Combining this with the previous lower bound for the tail product, we obtain
    \[
    V(z)
    \geq
    \frac{1}{2} \prod_{5\leq p < 10^8} \left(1-\frac{2}{p}\right) \prod_{10^8\leq p<z} \left(1-\frac{2}{p}\right)
    >
    0.0036797\, \frac{e^{5.821}}{\log^2 z}.
    \]
    Since
    \[
    0.0036797\,e^{5.821}
    =
    1.241196\ldots
    >
    1.241,
    \]
    it follows that
    \[
    V(z)
    \geq
    \frac{1.241}{\log^2 z}.
    \]
\end{proof}

It remains to control the accumulated error term appearing in the lower-bound sieve. The preceding lemmas establish the dimension condition and the required lower bound for the main-term factor \(V(z)\) in the range \(z\geq z_1\). The following lemma controls \(R_4(\mathcal A,D)\) when the sieve is taken at level \(D=z^{14.66}\). Together with Lemmas~\ref{lem:K-presieved} and~\ref{lem:V-presieved}, this supplies the analytic estimates needed for the range \(N\geq N_1\) in the proof of Theorem~\ref{thm:main}.

\begin{lemma}\label{lem:R-presieved}
    For \(z\geq z_1\), we have
    \[
    R_4(\mathcal A,z^{14.66})
    <
    0.563\,z^{14.66}\log^8 z.
    \]
\end{lemma}

\begin{proof}
    By the same residue-class argument as in the proof of Theorem~7.2 of Nathanson~\cite{nathanson1996additive}, the remainder terms associated with the auxiliary sequence \(\mathcal B\) satisfy
    \[
    |\rho_\ell|\leq 2^{\omega(\ell)}
    \]
    for every square-free \(\ell\). Moreover, \eqref{eq:presieved-size} together with the explicit formula for \(X\) gives
    \[
    |\rho_3| \leq \frac23.
    \]
    
    Let \(d\mid P(z)\) with \(d>1\). Since \(3\nmid d\), we have \(g(d)\leq1/2\), and hence \eqref{eq:presieved-remainder} gives
    \begin{align*}
    |r_d|
    &\leq
    |\rho_d|+|\rho_{3d}|+g(d)|\rho_3|\\
    &\leq
    2^{\omega(d)}+2^{\omega(3d)}+\frac13\\
    &=
    3\cdot2^{\omega(d)}+\frac13\\
    &\leq
    \frac{19}{6}\,2^{\omega(d)}.
    \end{align*}
    Here the last inequality follows from \(2^{\omega(d)}\geq2\). For \(d=1\), we have \(r_1=0\).
    
    Since \(d\mid P(z)\) is square-free,
    \[
    \tau_4(d)2^{\omega(d)}=\tau_8(d).
    \]
    Therefore, for \(0<\delta<1\), Rankin's trick gives
    \begin{align}
        R_4(\mathcal A,D)
        &=
        \sum_{\substack{d\mid P(z)\\d<D}} \tau_4(d)|r_d| \notag\\
        &\leq
        \frac{19}{6} \sum_{\substack{d\mid P(z)\\d<D}} \tau_8(d) \notag\\
        &\leq
        \frac{19}{6}D^\delta \sum_{d\mid P(z)} \frac{\tau_8(d)}{d^\delta} \notag\\
        &=
        \frac{19}{6}D^\delta \prod_{\substack{p<z\\ p\in\mathcal P}} \left(1+\frac8{p^\delta}\right). \label{eq:R4-presieved-rankin}
    \end{align}

    Using \(1+x<e^x\) for \(x>0\), we obtain
    \[
    R_4(\mathcal A,D)
    <
    \frac{19}{6}D^\delta \exp\left(8\sum_{\substack{p<z\\ p\in\mathcal P}} \frac{1}{p^\delta} \right).
    \]
    
    Write \(\delta=1-\eta\) for \(0< \eta < 1\). Then
    \[
    \frac{1}{p^\delta}
    =
    \frac{1}{p}e^{\eta\log p}.
    \]
    
    Since \(p<z\), the inequality \(e^x<1+xe^x\) gives
    \[
    e^{\eta\log p}
    <
    1+\eta(\log p)z^\eta.
    \]
    Thus
    \begin{align*}
    R_4(\mathcal A,D)
    &<
    \frac{19}{6}D^{1-\eta} \exp\left(8\sum_{\substack{p<z\\ p\in\mathcal P}}\frac{1}{p}\right) \exp\left(8\eta z^\eta \sum_{\substack{p<z\\ p\in\mathcal P}}\frac{\log p}{p} \right).
    \end{align*}
    
    Theorem~5 and the corollary to Theorem~6 of Rosser and Schoenfeld~\cite{rosser_schoenfeld1962} imply that, for \(z\geq z_1\),
    \[
    \sum_{p<z}\frac{1}{p}
    <
    \log\log z+0.263,
    \qquad
    \sum_{p<z}\frac{\log p}{p}
    <
    \log z.
    \]
    Since the prime \(3\) is omitted, the first estimate gives
    \[
    \sum_{\substack{p<z\\ p\neq3}}\frac{1}{p}
    <
    \log\log z + 0.263 - \frac{1}{3}.
    \]
    Also,
    \[
    \sum_{\substack{p<z\\ p\neq3}}\frac{\log p}{p}
    <
    \log z.
    \]
    Hence
    \[
    R_4(\mathcal A,D)
    <
    \frac{19}{6}D^{1-\eta} \exp\left(8\log\log z + 8\left(0.263-\frac{1}{3}\right) + 8\eta z^\eta\log z\right).
    \]
    
    Now put
    \[
    \eta=\frac{\alpha}{\log z}.
    \]
    Then \(z^\eta=e^\alpha\), and if \(D=z^s\), then
    \[
    D^{1-\eta}
    =
    D\exp(-\alpha s).
    \]
    Therefore
    \[
    R_4(\mathcal A,D)
    <
    \frac{19}{6} D\log^8 z \exp\left(8\left(0.263-\frac{1}{3}\right) + 8\alpha e^\alpha - \alpha s\right).
    \]
    
    We now take \(s=14.66\), so that \(D=z^{14.66}\). The value of \(\alpha\) minimising the exponent satisfies
    \[
    8e^\alpha(1+\alpha)=14.66
    \]
    and is \(0.32458\dots\). We therefore take \(\alpha=0.3246\). Since \(z\geq z_1\), the corresponding value \(\eta=\alpha/\log z\) lies in \((0,1)\). A direct computation gives
    \[
    \frac{19}{6} \exp\left(8\left(0.263-\frac{1}{3}\right) + 8(0.3246)e^{0.3246} - 14.66(0.3246)\right)
    <
    0.563.
    \]
    Thus
    \[
    R_4(\mathcal A,z^{14.66})
    <
    0.563\,z^{14.66}\log^8 z,
    \]
    as claimed.
\end{proof}

\section{Proof of Theorem~\ref{thm:main}}\label{sec:proof-main}
We begin by disposing of a finite initial range of \(N\). This follows from known explicit maximal prime-gap data~\cite{caldwell_etal_maximal_gaps}. The point of the following lemma is that, if \(N/m\) lies in a range where prime gaps are bounded, then choosing the largest prime \(p<N/m\) forces the complementary summand \(N-mp\) to be small. This gives a bound for \(\Omega(N-mp)\), while choosing \(m\) to be either \(1\) or a prime ensures that \(\Omega(mp)\leq2\).

Set
\[
G:=101412319996363310923.
\]
The table of known maximal prime gaps~\cite{caldwell_etal_maximal_gaps} shows that every prime \(q<G\) is followed by a prime \(q'\) satisfying
\[
q'-q\leq1854.
\]

\begin{lemma}\label{lem:smallN}
    Let \(m\) be either \(1\) or a prime, and let \(M\) be a positive integer. Suppose
    \[
    1854m\leq 2^{M-2}.
    \]
    Then for every
    \[
    2\leq N\leq Gm
    \]
    there exist positive integers \(a,b\) such that
    \[
    N=a+b,
    \qquad
    \Omega(ab)\leq M.
    \]
\end{lemma}

\begin{proof}
    We first handle the case \(N\leq2m\). Taking \(a=1\) and \(b=N-1\), we have
    \[
    \Omega(ab)=\Omega(N-1)\leq \log_2(N-1)<\log_2(1854m)\leq M-2<M,
    \]
    so the result follows in this case.
    
    We may therefore assume that \(N>2m\). Let \(p\) be the largest prime strictly less than \(N/m\), and let \(p'\) be the next prime after \(p\). Since
    \[
    2<\frac Nm\leq G,
    \]
    we have \(p<G\). Hence the table of maximal prime gaps gives
    \[
    p'-p\leq1854.
    \]
    As \(p<N/m\leq p'\), we have
    \[
    0<N-mp\leq m(p'-p)\leq1854m.
    \]
    So
    \[
    \Omega(N-mp)
    \leq
    \log_2(N-mp)
    \leq
    \log_2(1854m)
    \leq
    M-2.
    \]
    Now set
    \[
    a=mp,
    \qquad
    b=N-mp.
    \]
    Then \(a\) and \(b\) are positive integers, \(N=a+b\), and since \(m\) is either \(1\) or a prime and \(p\) is prime,
    \[
    \Omega(mp)\leq2.
    \]
    Therefore
    \[
    \Omega(ab)
    =
    \Omega(mp)+\Omega(N-mp)
    \leq
    2+(M-2)
    =
    M.
    \]
\end{proof}

The following lemma converts a positive lower bound for the sifted sum into a bound for \(\Omega(ab)\). The factor \(2\) arises from the quadratic upper bound \(a_n<N^2\), while the normalisation \(a_n=n(N-n)/2\) in the odd case contributes one additional prime factor when passing from \(a_n\) to \(ab\).

\begin{lemma}\label{lem:sieve-to-almost-prime}
    Let \(N\geq 2\), let \(r>1\), and put
    \[
    z=N^{1/r}.
    \]
    If
    \[
    S(\mathcal A,z)>0,
    \]
    then there exist positive integers \(a,b\) such that \(N=a+b\) and
    \[
    \Omega(ab)<2r
    \]
    when \(N\) is even, while
    \[
    \Omega(ab)<2r+1
    \]
    when \(N\) is odd.
\end{lemma}

\begin{proof}
    Since \(S(\mathcal A,z)>0\), there exists \(n\in I_3\) such that
    \[
    (a_n,P(z))=1.
    \]
    By the definition of \(I_3\), we also have \(3\nmid a_n\). Since \(\mathcal P\) contains every prime other than \(3\), these two conditions imply that \(a_n\) has no prime factor less than \(z\).
    
    Set
    \[
    t:=\Omega(a_n).
    \]
    Since every prime factor of \(a_n\) is at least \(z\), counted with multiplicity, we have
    \[
    z^t\leq a_n.
    \]
    On the other hand, in both parity cases,
    \[
    a_n
    \leq
    n(N-n)
    \leq
    \frac{N^2}{4}
    <
    N^2.
    \]
    Consequently,
    \[
    N^{t/r}
    =
    z^t
    \leq
    a_n
    <
    N^2.
    \]
    This implies
    \[
    \frac{t}{r}<2,
    \]
    and hence
    \[
    \Omega(a_n)=t<2r.
    \]
    
    Now set
    \[
    a:=n,
    \qquad
    b:=N-n.
    \]
    Then \(a\) and \(b\) are positive integers and \(N=a+b\). If \(N\) is even, then
    \[
    ab=a_n,
    \]
    so
    \[
    \Omega(ab)=\Omega(a_n)<2r.
    \]
    If \(N\) is odd, then
    \[
    ab=2a_n,
    \]
    and therefore
    \[
    \Omega(ab)=\Omega(a_n)+1<2r+1.
    \]
\end{proof}

\begin{proof}[Proof of Theorem~\ref{thm:main}]
    We use the three ranges specified at the beginning of Section~\ref{sec:estimates}.
    
    \medskip
    
    \noindent
    \textbf{Case 1: \(2\leq N\leq N_0\).}
    
    Apply Lemma~\ref{lem:smallN} with \(m = 1158293\) (which is prime) and \(M=33\).
    Indeed,
    \[
    1854m
    =
    2147475222
    <
    2147483648
    =
    2^{31}
    =
    2^{M-2},
    \]
    and
    \[
    Gm
    =
    117465180365547648498934439
    =
    N_0.
    \]
    The result therefore follows from Lemma~\ref{lem:smallN}.
    
    \medskip

    \noindent
    \textbf{Case 2: \(N_0<N<N_1\).}
    
    In this range,
    \[
    z_0<z=N^{1/16.5}<z_1.
    \]
    We shall prove that
    \[
    S(\mathcal A,z)>0
    \]
    by finite verification. Rather than treating individual values of \(N\), we divide \([z_0,z_1]\) into finitely many intervals and obtain a lower bound that is uniform on each interval.
    
    By Lemma~\ref{lem:K-presieved}, the dimension condition holds in this range with \(\kappa=2\) and \(K=1.146\). Put
    \[
    k:=2+\log(1.146)
    \qquad\text{and}\qquad
    s:=14.8.
    \]
    Then \(s\geq2k+3\), and \(F(s,k)\) is positive. Let \(0<\delta<1\) and set \(D=z^s\). If
    \[
    z_0\leq u<v\leq z_1 \qquad\text{and}\qquad z\in(u,v],
    \]
    then
    \[
    X>\frac{u^{16.5}-1}{3}
    \]
    and
    \[
    V(z)\geq V(v)
    \geq
    \frac12\prod_{5\leq p<v}\left(1-\frac2p\right).
    \]
    Taking \(D=z^s\) in \eqref{eq:R4-presieved-rankin} and using
    \(z\leq v\), we obtain
    \[
    2R_4(\mathcal A,z^s)
    \leq
    \frac{19}{3}v^{s\delta} \prod_{\substack{p<v\\p \in \mathcal P}}\left(1+\frac8{p^\delta}\right).
    \]
    Therefore Theorem~\ref{thm:FI} gives
    \[
    S(\mathcal A,z)>S(u,v;s,\delta),
    \]
    where
    \[
    S(u,v;s,\delta)
    :=
    \frac{u^{16.5}-1}{6}
    \prod_{5\leq p<v}\left(1-\frac{2}{p}\right) \left\{1-\frac{s+3}{2e^k}\left(\frac{2ek}{s-3}\right)^{(s-3)/2}\right\}
    -
    \frac{19}{3} v^{s\delta} \prod_{\substack{p<v\\p \in\mathcal P}} \left(1+\frac{8}{p^\delta}\right).
    \]
    
    We now construct a finite sequence
    \[
    z_0=q_0<q_1<\cdots<q_j = z_1
    \]
    such that \(q_1, \dots q_{j-1}\) are prime. In our verification, the candidate endpoints are consecutive primes near \(z_0\), and their spacing is progressively increased, with every \(10{,}000\)th prime used once the candidate endpoint is at least \(10^7\). We fix \(s=14.8\) and write
    \[
    \Delta:=\{0.20,0.60,0.75,0.85,0.93\}.
    \]
    For each interval \((q_i,q_{i+1}]\), the finite verification establishes
    that
    \[
    \max_{\delta\in\Delta}
    S(q_i,q_{i+1};s,\delta)>0.
    \]
    The sequence \((q_i)\) and the verification of these inequalities are reproduced by the script \path{verify_case2_presieved.py} in the verification repository~\cite{VerificationRepo}. Therefore, for any \(z\in(z_0,z_1]\), choosing \(i\) such that
    \[
    z\in(q_i,q_{i+1}]
    \]
    gives
    \[
    S(\mathcal A,z)
    >
    \max_{\delta\in\Delta}
    S(q_i,q_{i+1};s,\delta)
    >
    0.
    \]
    
    Hence \(S(\mathcal A,z)>0\) for every \(z\in(z_0,z_1]\). Applying Lemma~\ref{lem:sieve-to-almost-prime} with \(r=16.5\), there exist positive integers \(a,b\) with \(N=a+b\) and \(\Omega(ab)<33\) when \(N\) is even, and \(\Omega(ab)<34\) when \(N\) is odd. As \(\Omega(ab)\) is an integer, \(\Omega(ab)\leq 33\) in both cases.
    
    \medskip
    
    \noindent
    \textbf{Case 3: \(N\geq N_1=10^{132}\).}
    
    We apply Theorem~\ref{thm:FI} to the pre-sieved sequence with
    \[
    z=N^{1/16.5},\qquad D=z^{14.66},\qquad
    s=14.66,\qquad k=2+\log(1.097).
    \]
    The condition \(s\geq2k+3\) is satisfied, and a direct computation gives
    \[
    1 - \frac{s+3}{2e^k} \left(\frac{2ek}{s-3}\right)^{(s-3)/2}
    >
    0.0563.
    \]
    Since
    \[
    X\geq\frac{N-1}{3},
    \]
    Lemmas~\ref{lem:K-presieved}, \ref{lem:V-presieved}, and \ref{lem:R-presieved} give
    \begin{align*}
    S(\mathcal A,z)
    &>
    \frac{1.241(N-1)}{3\log^2z}
    (0.0563)
    -
    2(0.563)z^{14.66}\log^8z \\
    &=
    \frac{1.241(16.5)^2(0.0563)}{3}
    \frac{N-1}{\log^2N}
    -
    \frac{1.126}{16.5^8}
    N^{14.66/16.5}\log^8N \\
    &>
    6.34\frac{N-1}{\log^2N}
    -
    2.05\times10^{-10}
    N^{14.66/16.5}\log^8N.
    \end{align*}
    The ratio of the first term in the final expression to the second is
    \[
    \frac{6.34}{2.05\times10^{-10}}
    \left(1-\frac1N\right)
    \frac{N^{1.84/16.5}}{\log^{10}N}.
    \]
    This ratio is increasing for \(N\geq10^{132}\), and at \(N=N_1\) a direct computation shows that it exceeds \(2.41\). Hence
    \[
    S(\mathcal A,z)>0.
    \]
    Applying Lemma~\ref{lem:sieve-to-almost-prime} with \(r=16.5\), there exist positive integers \(a,b\) with \(N=a+b\) such that
    \[
    \Omega(ab)<33
    \]
    when \(N\) is even, while
    \[
    \Omega(ab)<34
    \]
    when \(N\) is odd. Since \(\Omega(ab)\) is an integer, it follows that
    \[
    \Omega(ab)\leq33
    \]
    in both cases.
\end{proof}

\section{Further work}\label{sec:future-work}

The numerical results obtained in the proof suggest that the present implementation has not yet exhausted the method. In particular, a bound of \(\Omega(ab)\leq32\) appears to be within reach, while the stronger bound \(\Omega(ab)\leq31\) remains a credible possibility. Possible refinements include:
\begin{itemize}
    \item calculating the constant \(K\) in the dimension condition separately on each interval in the finite verification range;
    \item optimising both \(s\) and \(\delta\) on each interval, rather than fixing \(s\) throughout the computation and restricting \(\delta\) to a short predetermined list;
    \item pre-sieving by both \(3\) and \(5\), thereby also removing the prime \(5\), which is responsible for the largest values of \(K\) in the present calculation, at the expense of a more careful treatment of the resulting remainder terms;
    \item sharpening the constants in the accumulated error term and the accompanying product estimates.
\end{itemize}
These improvements are complementary. In particular, a smaller interval-dependent value of \(K\) permits a smaller choice of \(s\), which in turn reduces the exponent appearing in the accumulated error term and helps to keep the finite verification computationally manageable. With all other constants unchanged, if \(r\) were reduced from \(16.5\) to \(16\), the lower bound used in Case~3 of the proof of Theorem~\ref{thm:main} would become positive only once \(N\) is of order \(10^{192}\). The finite verification in Case~2 would therefore need to extend to \(z\) of order \(10^{12}\), making the computation substantially more demanding.

A further improvement may be possible through a stronger explicit lower-bound sieve factor than the Friedlander--Iwaniec expression \(F(s,k)\) used here. It is not clear, however, whether a useful explicit improvement is available within the same sieve framework, or how much such an improvement would contribute relative to the refinements listed above.

More substantial reductions in the number of prime factors would likely require a different sieve framework. One natural direction is the use of weighted sieves. For example, explicit forms of Kuhn's weighted sieve and Richert's weighted sieve have recently been used in the study of almost primes between consecutive squares~\cite{dudek_johnston2026_squares} and primes and almost primes between consecutive cubes~\cite{johnston_sorenson_thomas_webster2026}, respectively. Such weighted sieves may exploit information about the sizes of the remaining prime factors more efficiently than the unweighted lower-bound sieve used here. They may therefore offer a route to further reductions, potentially including bounds below \(30\), although adapting them explicitly to the present problem would require significant additional work.

\section*{Acknowledgements}

The author thanks Adrian W. Dudek for helpful discussions and comments on an earlier version of this manuscript. The author is supported by an Australian Government Research Training Program Scholarship.

\bibliographystyle{abbrv}
\bibliography{refs}

\end{document}